\documentclass[preprint,10pt]{elsarticle}

\usepackage{amsmath,amsopn,amssymb,epsfig} 
\usepackage{hyperref}
\usepackage{color}

\newtheorem{Theorem}{Theorem}[section]
\newtheorem{Lemma}{Lemma}[section]

\newtheorem{example}{Example}[section]

\newtheorem{Corollary}{Corollary}[section]
\newproof{proof}{Proof}
\newproof{pot}{Proof of Theorem \ref{thm2}}

\def\r2n{\mathbb{R}^{2n}}
\def\r2n2n{\mathbb{R}^{2n\times 2n}}

\def\cnn{\mathbb{C}^{n\times n}}

\newcommand{\bb}{\begin{bmatrix}}
\newcommand{\eb}{\end{bmatrix}}

\begin{document}

\date{}
\begin{frontmatter}
\title{
The Eigenvalue Shift Technique and Its Eigenstructure Analysis of a Matrix
\tnoteref{t1}}

\author{Chun-Yueh Chiang\fnref{fn1}}
\ead{chiang@nfu.edu.tw}
\address{Center for General Education, National Formosa
University, Huwei 632, Taiwan.}
\author{Matthew M. Lin \corref{cor1}\fnref{fn2}}
\ead{mlin@math.ccu.edu.tw}
\address{Department of Mathematics, National Chung Cheng University, Chia-Yi 621, Taiwan.}
\cortext[cor1]{Corresponding author}
\fntext[fn1]{The first
author was supported  by the National Science Council of Taiwan under grant
NSC100-2115-M-150-001.}
\fntext[fn2]{The second
author was supported by the National Science Council of Taiwan under grant
NSC99-2115-M-194-010-MY2.}

\date{ }

\begin{abstract}

The eigenvalue shift technique is the most well-known and fundamental tool for matrix computations. Applications include the search of eigeninformation, the acceleration of numerical algorithms, the study of Google's PageRank. The shift strategy arises from the concept investigated by Brauer~\cite{Brauer52} for changing the value of an eigenvalue of a matrix to the desired one, while keeping the remaining eigenvalues and the original eigenvectors unchanged. The idea of shifting  distinct eigenvalues can easily be generalized by Brauer's idea. However, shifting an eigenvalue with multiple multiplicities is a challenge issue and worthy of our investigation. In this work,  we propose a new way for updating an eigenvalue with multiple multiplicities and thoroughly analyze its corresponding Jordan canonical form after the update procedure.

\end{abstract}


\end{frontmatter}

\section{Introduction}

The eigenvalue shift technique is a will-established method of mathematics in science and engineering. It arises from the research of  eigenvalue computations.
As we know,  the power method is the most common and easiest algorithm to find the dominant eigenvalue  for a given matrix $A\in\mathbb{C}^{n\times n}$, but it is impracticable to look for a specific eigenvalue of $A$.
In order to compute one specific eigenvalue, we need to apply the power method to the inverse of the shifted matrix $(A-\mu I_n)$ with an appropriately chosen shift value $\mu\in\mathbb{C}$.
Here, $I_n$ is an $n\times n$ identity matrix.
This is the well-known shifted inverse power method~\cite{Golub96}. Early applications of eigenvalue shift techniques are focused on the stability analysis of dynamical systems~\cite{Bellman97}, the computation of the root of a linear or non-linear equation by using preconditioning approaches~\cite{Kelley95}. Only recently, the study of eigenvalue shift approaches has been widely applied to the study of inverse eigenvalue problems for reconstructing a structured model from prescribed spectral data~\cite{Chu02,Chu05,Chu09,Chu10}, the structure-preserving double algorithm for solving nonsymmetric algebraic matrix Riccati equation
~\cite{GuoBruMei2007,Bini2008}, and the Google's PageRank problem~\cite{Horn06, Wu08}.  Note that the common goal of the
applications of the shift techniques stated above is
to replace some unwanted eigenvalues so that the stability or acceleration of the prescribed algorithms can be obtained.  Hence, a natural idea, the main contribution of our work, is to propose a method for updating the eigenvalue of a given matrix and provide the precise Jordan structures of this updated matrix.

For a matrix $A\in\mathbb{C}^{n\times n}$ and a number $\mu$, if $(\lambda,v)$ is an eigenpair of $A$, then matrices $A+\mu I_n$ and $\mu A$ have the same eigenvector $v$ corresponding to the eigenvalue $\lambda+\mu$ and $\mu \lambda$, respectively. However, these two processes are to translate or scale \emph{all} eigenvalues of A.
Our
work here is to change a particular eigenvalue, which is  enlightened through the following important result given by Brauer~\cite{Brauer52}.
%
%
\begin{Theorem}{\bf (Brauer).}\label{lem:Brauer}
Let $A$ be a  matrix with $Av = \lambda_0 v$ for some nonzero vector $v$. If $r$ is a vector so that $r^{\top}v =1$, then for any scalar $\lambda_1$, the eigenvalues of the matrix
\begin{equation*}
\widehat{A} = A+ (\lambda_1-\lambda_0) v r^{\top},
\end{equation*}
consist of those of $A$, except that one eigenvalue $\lambda_0$ of $A$ is replaced by $\lambda_1$. Moreover, the eigenvector $v$ is unchanged, that is, $\widehat{A} v= \lambda_1 v$.
\end{Theorem}

To demonstrate our motivation,  consider an $n\times n$ complex matrix $A$.
%
Let   $ A=V J V^{-1}$ be a Jordan matrix  decomposition of $A$, where $J$ is the Jordan normal form of $A$ and  V is a matrix consisting of generalized eigenvectors of $A$. Denote the matrix $J$ as
\begin{equation*}
J= \left[\begin{array}{cc}J_{k}(\lambda) & 0 \\0 & J_{2}\end{array}\right],
\end{equation*}
where
$J_{k}(\lambda)$ is a $k\times k$  Jordan block corresponding to eigenvalue $\lambda$ given by
\begin{equation*}
J_{k}(\lambda) := \left[\begin{array}{ccccc}
       \lambda & 1 & 0 & \cdots & 0 \\ 0 & \lambda & \ddots & \ddots & \vdots \\
       \vdots & \ddots & \ddots & \ddots & 0 \\ \vdots & & \ddots & \lambda & 1 \\
       0 & \cdots & \cdots & 0 & \lambda
    \end{array}\right]\in\mathbb{C}^{k\times k}
\end{equation*}
and  $J_{2}$ is an $(n-k)\times(n-k)$ matrix composed of
 a finite direct sum of Jordan blocks.
Then, partition matrices $V$ and $V^{-1}$ conformally as
\begin{equation*}
V= \left[\begin{array}{cc}V_1 & V_2\end{array}\right], \,
(V^{-1})^*= \left[\begin{array}{cc}V_3 & V_4\end{array}\right],
\end{equation*}
with $V_1,V_3\in\mathbb{C}^{n\times k}$ and $V_2, V_4\in\mathbb{C}^{n\times (n-k)}$.
A natural idea of updating the eigenvalue appearing at the top left corner of the matrix $J$, but keeping $V$ and $V^{-1}$ unchanged, is to add a rank-$k$ matrix
 $\Delta A=V_1 \Delta J_{1} V_3^*$
 to the original matrix $A$ so that
  \begin{equation*}
 A+\Delta A=V
 \left[\begin{array}{cc} J_{k}(\lambda)+\Delta J_{1} &  0 \\ 0&  J_{2}\end{array}\right] V^{-1}
 \end{equation*}
 has the desired eigenvalue.
 Here,  $\Delta J_{1}$ is a particular $k\times k$ matrix, which makes  $J_{k}(\lambda)+\Delta J_{1}$ a new Jordan matrix.
Observe that the updated matrix $A+\Delta A$ preserves the structures of matrices $V$ and $V^{-1}$, but changes the unwanted eigenvalue via the new Jordan matrix  $J_{k}(\lambda)+\Delta J_{1}$.   In other words, if we know the Jordan matrix decomposition in prior,  we can update the eigenvalue of $A$ without no difficulty.
Note that  $V_3^*$ and $V_1$ are matrices composed of the generalized left and right eigenvectors, respectively, and $V_3^*V_1 = I_k$.    In practice, given any generalized left and right eigenvectors corresponding to $J_k(\lambda)$, the condition $V_3^*V_1 = I_k$ is not true in general. Hence, the eigenvalue shift approach stated above cannot not be applied.

In this paper, we want to investigate an eigenvalue shift technique for updating the eigenvalue of the matrix $A$, provided that  partial generalized left and right eigenvectors are given. We show that after the shift technique the first half generalized eigenvectors are kept the same. Indeed, the study of the invariant of the first half generalized eigenvectors is essential for finding the stabilizing solution of an algebraic Riccati Equation and is the so-called \emph{Schur method}  or \emph{invariant subspace method}~\cite{Laub1979, Guo2006}.
To advance our research we organize this paper as follows. In Section 2,  several useful features of  the generalized  left and right eigenvectors are discussed. In particular, we investigate the principle of generalized biorthogonality of generalized eigenvectors. This principle is then applied to the study
of the eigenvalue shift technique in Section 3 and 4. Finally, the conclusions and some open problems are given in Section 5.

\section{The Principle of Generalized Biorthogonality
}
%
For a given $n\times n$ square matrix $A$, let $\sigma(A)$ be the set of all eigenvalues of $A$. We say that two vectors $u$ and $v$ in $\mathbb{C}^{n}$ are orthogonal if $u^*v = 0$. In our study, we are seeking the orthogonality of eigenvectors of a given matrix.
The feature is know as the \emph{principle of biorthogonality}
and is discussed in~\cite[Theorem~1.4.7]{HornJohnson90} as follows:

\begin{Theorem}\label{thmbio0}
If $A\in\mathbb{C}^{n\times n}$ and if $\lambda_1,\lambda_2\in\sigma(A)$, with $\lambda_1\neq \lambda_2$, then any left eigenvector of $A$ corresponding to $\lambda_1$ is orthogonal to
any right eigenvector of $A$ corresponding to $\lambda_2$.
\end{Theorem}

Theorem~\ref{thmbio0} tells us the principle of biorthogonality
with respect to any two distinct eigenvalues. Next, we want to  enhance this feature  to generalized left and right eigenvectors.



%
%
\begin{Theorem}\label{thm:bio}
[Generalized Biorthogonality Property] 
Let $A\in\cnn$ and let $\lambda_1, \lambda_2\in\sigma(A)$. Suppose  $\{u_i\}_{i=1}^{p}$
and $\{v_i\}_{i=1}^{q}$  are the generalized left and right eigenvectors corresponding to the Jordan block $J_{p}(\lambda_1)$ and $J_{q}(\lambda_2)$, respectively. Then
 \begin{itemize}
 \item[(a)]
 If $\lambda_1\neq \lambda_2$,
 \begin{equation*}
u_{i}^* v_{j}=0, \quad 1 \leq i\leq p, \,  1\leq j \leq q.
\end{equation*}
%
%
 \item[(b)]
 If $\lambda_1 = \lambda_2$,
 \begin{subequations}
\begin{eqnarray}\label{eq:biothogonal}
u_{i}^* v_{j}&=& u_{i-1}^* v_{j+1}, \quad 2 \leq i\leq p, \,  1\leq j \leq q-1,\\
u_{i}^* v_{j}&=&0,   \quad 2 \leq i+j\leq \max\{p,q\} .\end{eqnarray}
\end{subequations}
\end{itemize}
\end{Theorem}
\begin{proof}
Let $U$ and $V$ be two matrices defined by
\begin{equation*}
U = [\begin{array}{ccc} u_1 & \ldots & u_p\end{array} ] \in \mathbb{C}^{n\times p}, \,
V =  [\begin{array}{ccc} v_1 & \ldots & v_q\end{array} ] \in \mathbb{C}^{n\times q}.
\end{equation*}
By the definitions of $\{u_i\}_{i=1}^{p}$
and $\{v_i\}_{i=1}^{q}$,  we have
\begin{equation*}
U^* A= J_{p}^\top({\lambda}_1) U^* ,\,
AV   =  VJ_{q}(\lambda_2).
\end{equation*}
That is,
\begin{equation}\label{eq:biouv}
U^* A V=(U^* V)J_{q}(\lambda_2)=J_{p}^\top({\lambda}_1) (U^* V).
\end{equation}
Define components $x_{i,j} = u^*_i v_j$, $x_{i,0} = 0$,  and $x_{0,j} = 0$, for $i = 1,\ldots, p$ and $j = 1,\ldots, q$. Then~\eqref{eq:biouv} implies that
\begin{equation}\label{eq:x}
(\lambda_1-\lambda_2)x_{i,j}= x_{i-1,j} -x_{i,j-1},\quad 1\leq i \leq p,\,1\leq j\leq q.
\end{equation}
It follows from~\eqref{eq:x} and the assumptions of $x_{i,j}$  that if  $\lambda_1\neq \lambda_2$, then
$x_{i,j} = 0$ for $1\leq i\leq p$ and $ 1 \leq j \leq q $. This proves (a).

By~\eqref{eq:x}, we see that
if  $\lambda_1=\lambda_2$,
$x_{i-1,j}=x_{i,j-1}$ for $1\leq i\leq p$ and $1\leq j\leq q$.
It follows that
$x_{i,j} = 0$, for $2 \leq i+j \leq \max\{p,q\}$ and
elements $x_{i,j}$ coincide on the "matrix diagonals"
$i+j = s$ for any $\max\{p,q\} + 1 \leq s \leq p+q$. This completes the proof of (b). \hfill $\Box$
\end{proof}
Note that the values  $x_{i,j}$ defined in the proof of Theorem~\ref{thm:bio} imply that if $\lambda_1=\lambda_2$,  the matrix  $X = [x_{i,j}]_{p \times q}$ is indeed a \emph{lower triangular Hankel matrix}. Also, by~\eqref{eq:biothogonal},
we have the following useful results.
\begin{Corollary}\label{Cor:bio}
Let $A\in\cnn$ and let $\lambda\in\sigma(A)$. Suppose  $\{u_i\}_{i=1}^{p}$
and $\{v_i\}_{i=1}^{q}$  are the generalized left and right eigenvectors corresponding to the Jordan block $J_{p}(\lambda)$ and $J_{q}(\lambda)$, respectively. Then,
if $p$ and $q$ are even, then
\begin{equation}\label{eq:bioeven}
u_{i}^* v_{j}=0, \quad 1 \leq i\leq \frac{p}{2}, \,  1\leq j\leq \frac{q}{2};
\end{equation}
if $p$ and $q$ are odd, then
\begin{equation}\label{eq:bioodd}
u_{i}^* v_{j}=0, \, u^*_{\frac{p+1}{2}} v_j = 0, \, u^*_{i} v_{\frac{q+1}{2}} = 0, \quad 1\leq i\leq \frac{p-1}{2}, \, 1 \leq j\leq \frac{q-1}{2}.
\end{equation}
\end{Corollary}

Note that Corollary~\ref{Cor:bio} provides only the necessary conditions for two generalized left and right eigenvectors to be orthogonal. It is possible that two generalized eigenvectors are orthogonal, even if they are not fitted in the
constraints given in~\eqref{eq:bioeven} and~\eqref{eq:bioodd}.
This phenomenon can be observed by the following two examples. We first provide all possible types of generalized eigenvectors of a Jordan matrix with one and two Jordan blocks, respectively.  We then come up with two extreme cases for each Jordan matrix under discussion. One is with the smallest number of orthogonal generalized left and right eigenvectors. The other is with the largest number of orthogonal ones.

\begin{example}\label{ex1}
Let $A=J_{2k}(\lambda)$. Then all of the generalized left and right eigenvectors $\{u_i\}_{i=1}^{2k} \subset\mathbb{C}^{2k}$ and $\{v_i\}_{i=1}^{2k}\subset\mathbb{C}^{2k}$, respectively, can be written as
\begin{align*}
u_i=\sum\limits_{j=1}^i a_{i-j+1} e_{2k-j+1}, \,
v_i=\sum\limits_{j=1}^i b_{i-j+1} e_j,
\,1\leq i \leq 2k,
\end{align*}
where $a_i$ and $b_i$ are  arbitrary complex numbers and $a_1b_1\neq 0$. Moreover, if
$a_i=b_i=1$ for all $i$, then
\begin{align*}
&u_i^*v_j = 0 ,\,1\leq i,j \leq k,\\
&u_i^* v_j\neq 0 ,\,k+1\leq i,j \leq 2k,
\end{align*}
and if
$a_1=b_1=1$ and $a_i=b_i=0$ for all $i>1$, then
\begin{align*}
&u_i^*v_j  = 0,\,i+j\neq 2k+1,\\
&u_i^* v_j\neq 0 ,\,i+j= 2k+1.
\end{align*}
\end{example}

The second example demonstrates the orthogonal properties of the generalized eigenvectors between two Jordan blocks.

\begin{example}\label{ex2}
Let $A=J_{k}(\lambda)\oplus J_{k}(\lambda)$. Then all of the generalized left eigenvectors of the first Jordan block $J_k(\lambda)$ (the upper left corner) and  the right generalized eigenvectors of the second Jordan block $J_k(\lambda)$  (the lower right corner) can be written as
\begin{align*}
u_i=\sum\limits_{j=1}^i a_{j} e_{k-j+1}+b_{j} e_{2k-j+1},\,v_i=\sum\limits_{j=1}^i c_{i-j+1} e_j+d_{i-j+1} e_{k+j},\,1\leq i \leq k,
\end{align*} respectively, where $a_i,b_i,c_i$ and $d_i$ are  arbitrary complex numbers and $(|a_1|^2+|b_1|^2)(|c_1|^2+|d_1|^2)>0$. Moreover, if  $a_i=b_i=c_i=d_i=1$ for all $i$, then for all $1\leq i,j \leq k$
\begin{align*}
&u_i^* v_j = 0 ,\,i+j<k+1,\\
&u_i^* v_j\neq 0 ,\,i+j \geq k+1,
\end{align*}
and if $a_i=d_i=1,b_i=c_i=0$ for all $i$, then
\begin{align*}
&u_i^*v_j = 0 ,\,1\leq i,j \leq k.
\end{align*}
\end{example}

Given a matrix $A\in\mathbb{C}^{n\times n}$,  let us close this section with the study of  the mapping of the resolvent operator $(A-\lambda I_n)^{-1}$ on its generalized left and right eigenvectors.  Evidently, the orthogonal property between two generalized left and right eigenvectors will be influenced after the mapping. This influence will play a crucial role in proposing an eigenvalue shift technique later on.
%
%
\begin{Theorem}\label{lemA1}
Let $A$ be a matrix in $\mathbb{C}^{n\times n}$
and let $\lambda_0 \in \sigma(A)$.
Suppose  $\{u_i\}_{i=1}^p$ and $\{v_i\}^p_{i=1}$ be the generalized left and right eigenvectors corresponding to $J_p(\lambda_0)$. Let $\lambda$ be a complex number and $\lambda\not\in\sigma(A)$.
Then
\begin{enumerate}
\item[(a)] For $1\leq i \leq  p$,
\begin{eqnarray*}
(A-\lambda I_n)^{-1}v_i &=& \sum\limits_{j=1}^i \dfrac{(-1)^{i-j}v_j}{(\lambda_0-\lambda)^{i-j+1}},\\
u_i^*(A-\lambda I_n)^{-1} &=&\sum\limits_{j=1}^i \dfrac{(-1)^{i-j}u_i^*}{(\lambda_0-\lambda)^{i-j+1}}.
\end{eqnarray*}

\item[(b)]  For  $i+j\leq p$ and $i,j \geq 1$, $u_i^* (A-\lambda I_n)^{-1}v_j=0$.
\end{enumerate}
\end{Theorem}

\begin{proof}
Set $\lambda \in \mathbb{C}$ and $\lambda\not\in\sigma(A)$.  Since  $(A- \lambda I_n) v_1 = (\lambda_0- \lambda) v_1$, $(A- \lambda I_n) ^{-1} v_1 = \dfrac{v_1}{(\lambda_0 - \lambda)}$.  Also, $(A- \lambda I_n) v_2 = (\lambda_0 - \lambda) v_2 + v_1$.
It follows that
\begin{equation*}
(A- \lambda I_n) ^{-1}v_2 = \dfrac{v_2 -  (A- \lambda I_n) ^{-1}v_1 }{(\lambda_0 - \lambda)} = \sum\limits_{j=1}^2 \dfrac{(-1)^{2-j}v_j}{(\lambda_0-\lambda)^{2-j+1}}.
\end{equation*}
Subsequently, a similar proof can be given without difficulty to  different generalized  left and right eigenvectors, that is, (a) follows.

Apply Theorem~\ref{thm:bio} and part (a). It is true that  $u_i^* (A-\lambda I_n)^{-1}v_j=0$,  for $i+j \leq p$ and $i,j \geq 1$. The result of (b) is established.\hfill $\Box$
\end{proof}

Now we have enough tools to analyze the eigenvalue shift problems. We first establish a result for eigenvalues with even algebraic multiplicities and then generalize it to eigenvalues with odd algebraic multiplicities.  We show how to shift a desired eigenvalue of a given matrix without changing the remaining ones.

\section{Eigenvalue of Even Algebraic Multiplicities}
\label{sec:even}

In this section, we propose a shift technique to move an eigenvalue with even algebraic multiplicities  to a desired one. We claim that based on our approach, we can keep parts of generalized eigenvectors
unchanged.  We also investigate all of the possible Jordan structures after the shift at the end of this section.

\begin{Theorem}\label{lem:shiftA}
Let $A \in \mathbb{C}^{n\times n}$,  let $\lambda_0\in \sigma(A)$ with algebraic multiplicity $2k$ and geometric multiplicity $1$, and let
$\{u_i\}_{i=1}^{2k}$ and $\{v_i\}_{i=1}^{2k}$ be the left and right Jordan chains  for $\lambda_0$.
Define two matrices
\begin{equation}\label{eq:evenuv}
U : =\begin{bmatrix}
u_1 & u_2 &\cdots & u_k
\end{bmatrix},\,
V : =\begin{bmatrix}
v_1 & v_2 &\cdots & v_k
\end{bmatrix}.
\end{equation}
If matrices $R^*\in\mathbb{C}^{k\times n}$ and $L\in\mathbb{C}^{n\times k}$
are, respectively, the right and left inverses of $V$ and $U^*$ satisfying
\begin{equation}\label{eq:evenRU}
U^* L = R^* V=I_k,
\end{equation}
then the shifted matrix
\begin{equation}\label{widAeven}
\widehat{A} : =  A+(\lambda_1-\lambda_0)R_1R_2^*
\end{equation}
where
$
R_1 : =\begin{bmatrix}
V & L
\end{bmatrix}$ and
$R_2 :=
\begin{bmatrix}
R & U
\end{bmatrix},
$ has the following properties:
\begin{itemize}
\item[(a)]
The eigenvalues of $\widehat{A}$ consist of those of $A$, except that the eigenvalue $\lambda_0$
of $A$ is  replaced by $\lambda_1$.
\item[(b)]  $ \widehat{A}v_1=\lambda_1 v_1$, $u_1^* \widehat{A}=u_1^* \lambda_1$, $\widehat{A}v_{i+1}=\lambda_1 v_{i+1}+v_i$, $u_{i+1}^* \widehat{A}=\lambda_1 u_{i+1}^*+u_i^*$, for $i=1,\ldots,k-1$. That is,
\begin{align*}
\widehat{A} V &= VJ_{k}(\lambda_1),\\
U^* \widehat{A}&= J_{k}^\top(\lambda_1)U^*.
 \end{align*}

\end{itemize}
\end{Theorem}

\begin{proof}
This proof can be obtained by considering the characteristic polynomial of $\widehat{A}$, that is, if $\lambda\not\in\sigma(A)$, then
\begin{eqnarray}\label{charMhat}
  \det(\widehat{A}-\lambda I_{n}) &=&
    \det(A- \lambda I_n)
    \det\left( I_n +
    (\lambda_1 - \lambda_0) R_1R_2^*
    (A - \lambda  I_n)^{-1}
    \right) \\
   &=&     \det(A- \lambda I_n)
    \det\left(
    I_{2k} +
    (\lambda_1 - \lambda_0) R_2^*
    (A - \lambda  I_n)^{-1}
R_1\right) \nonumber\\
   & = &  \det(A- \lambda I_n) \left (
  \dfrac{\lambda_1 - \lambda}{\lambda_0 - \lambda }
   \right)^{2k}  (\mbox{ by Theorem~\ref{thm:bio} and Theorem~\ref{lemA1} }).\nonumber
 \end{eqnarray}

Since $\det(A -\lambda I_n)$ is a polynomial
with a finite number of zeros, we may choose a small perturbation $\epsilon > 0$ such that $\det(A -(\lambda+\epsilon) I_n) \neq 0$, that is, $A -(\lambda+\epsilon) I_n$ is invertible. This implies that if $\lambda\in\sigma(A)$, we can consider the characteristic polynomial
$\det(\widehat{A} -(\lambda+\epsilon) I_n)$ so that
\begin{eqnarray*}
  \det(\widehat{A}-(\lambda+\epsilon) I_{n}) &=&
 \det(A- (\lambda+\epsilon) I_n) \left (
  \dfrac{\lambda_1 - (\lambda+\epsilon)}{\lambda_0 - (\lambda+\epsilon) }
   \right)^{2k} .
\end{eqnarray*}
Observe that both sides of the above equation are continuous functions of $\epsilon$. By the so-called continuity argument method, letting $\epsilon\rightarrow 0$ gives that
\begin{eqnarray*}
  \det(\widehat{A}-\lambda I_{n}) &=&
 \det(A- \lambda I_n) \left (
  \dfrac{\lambda_1 - \lambda}{\lambda_0 - \lambda}
   \right)^{2k} .
\end{eqnarray*}
It follows that (a) holds.

To prove (b), we see that by Theorem~\ref{thm:bio},
\begin{eqnarray*}
 \widehat{A}v_1 \!&\!=\!&\!  Av_1+[(\lambda_1-\lambda_0)R_1R_2^* ] v_1
 = \lambda_0 v_1 + (\lambda_1 - \lambda_ 0) v_1
 = \lambda_1 v_1.\\
 \widehat{A}v_{i+1} \!&\!=\!&\! Av_i + [(\lambda_1-\lambda_0)
 R_1 R_2^*]v_i =   \lambda_1 v_{i+1} + v_i,  \quad i = 1,\ldots, k-1.
\end{eqnarray*}
The same approaches can be carried out to obtain
$u_1^* \widehat{A}=u_1^* \lambda_1$  and $u_{i+1}^* \widehat{A}=\lambda_1 u_{i+1}^*+u_i^*$, for $i=1,\ldots,k-1$. \hfill $\Box$
\end{proof}

Next, it is natural to ask about
the eigenstructure of this updated matrix $\widehat{A}$
defined in Theorem~\ref{lem:shiftA}.
To this end,  we use the notions given in Theorem~\ref{lem:shiftA}.
Assume without loss of generality that ${A}\in\mathbb{C}^{2k\times 2k}$, and that $P = \bb v_1& \ldots& v_{2k}\eb$ is a nonsingular matrix so that
\begin{eqnarray*}
P^{-1} {A} P &=& J_{2k}(\lambda_0),
\end{eqnarray*}
By Theorem~\ref{thm:bio}, we know that the matrix product
$\begin{bmatrix}
U^* \\ U_1^*
\end{bmatrix}
P\in\mathbb{C}^{2k \times 2k}$ is a lower triangular Hankel matrix, where
$U_1 = \bb u_{k+1} &\ldots& u_{2k}\eb$.  This suggests a way for constructing matrices $U$ and $V$ in~\eqref{eq:evenuv},
that is,
\begin{align*}
U = P^{-*}\begin{bmatrix} 0\\ Q\end{bmatrix}
\mbox{ and }
V = �P \begin{bmatrix}I_k \\ 0\end{bmatrix},
\end{align*}
where $Q$ is  some nonsingular and lower triangular Hankel matrix in $\mathbb{C}^{k\times k}$.
Also, the right and left inverses of~\eqref{eq:evenRU} can be
denoted by
\begin{equation*}
R =P^{-*}\begin{bmatrix} I_k \\ S_1^*\end{bmatrix},\,
L =P\begin{bmatrix} S_2 \\ Q^{-*} \end{bmatrix},
\end{equation*}
where $S_1$ and $S_2$ are arbitrary $k\times k$ matrices.
It follows from~\eqref{widAeven} that
\begin{align}\label{widhatAeven}
\widehat{A} &=  A+(\lambda_1-\lambda_0)(VR^*+LU^*)\\
            &=  P\left(  J_{2k}(\lambda_0)+(\lambda_1-\lambda_0)\left(
            \begin{bmatrix}I_k&S_1\\ 0&0 \end{bmatrix}+\begin{bmatrix}0 & S_2 Q^*\\0 & I_k \end{bmatrix}\right)\right)P^{-1}\nonumber\\
            &=   P \begin{bmatrix}J_{k}(\lambda_1) &(\lambda_1-\lambda_0)(S_1+S_2 Q^*)\\ 0&J_{k}(\lambda_1) \end{bmatrix}P^{-1}\nonumber
\end{align}
That is,
\begin{equation}\label{simeven}
\widehat{A} \sim T : =
\begin{bmatrix}J_{k}(\lambda_1) & C \\ 0&J_{k}(\lambda_1) \end{bmatrix},
\end{equation}
for some matrix $C\in \mathbb{C}^{k\times k}$.
Based on~\eqref{simeven}, the next two results discuss
the possible eigenspace and eigenstructure types of  the matrix $T$ (i.e., $\widehat{A}$). To facilitate our discussion below, we use $\{e_i\}_{1\leq i\leq k}$ to denote the standard basis in $\mathbb{R}^k$ from now on.

\begin{Lemma}\label{Eigeven}
Let $C$ be a matrix in $\mathbb{C}^{k\times k}$.
If $T$ is a matrix given by
\begin{eqnarray}\label{eq:evenAhat}
T &=& \left[\begin{array}{cc}J_k(\lambda) & C \\0 & J_k(\lambda)\end{array}\right],
\end{eqnarray}
then
$T$ has the eigenspace
\begin{align*}
E_\lambda=
\left\{
\begin{array}{rl}
\mbox{span}\left\{\bb e_1 \\ 0\eb,\bb 0 \\ e_1\eb \right\}, & \mbox{ if } c_{k,1}=0, \, Ce_1 = 0.    \smallskip\\
\mbox{span}\left\{\bb e_1 \\ 0\eb,\bb  N_k Ce_1 \\ e_1\eb\right\}, & \mbox{ if } c_{k,1} =0, \, Ce_1\neq0.   \smallskip\\
\mbox{span}\left\{\bb e_1 \\ 0\eb\right\}, & \mbox{ if }
c_{k,1} \neq0.
\end{array}
\right.
\end{align*}
 corresponding to the eigenvalue $\lambda$, where $N_k^\top = \lambda I_k - J_k(\lambda)$ and $C = [c_{i,j}]_{k\times k}$.

\end{Lemma}

\begin{proof}
Given two vectors $x_1$ and $x_2$ in $\mathbb{C}^{k\times 1}$, let  $v^\top =  \bb x_1^\top & x_2^\top \eb$ be an eigenvector of  $T$ with respect to $\lambda$.
It follows that
\begin{align*}
T v =\lambda v,
\end{align*}
or, equivalently,
\begin{align}\label{eq:Jordan}
\begin{array}{rcl}
(J_k(\lambda)-\lambda I_k) x_1&= &-Cx_2,\\
(J_k(\lambda)-\lambda I_k) x_2&=& 0.
\end{array}
\end{align}
The possible types of eigenspace can be obtained by directly solving~\eqref{eq:Jordan}.

\end{proof}

Lemma~\ref{Eigeven} suggests a possible characterization of all Jordan canonical forms of $T$.
Note that if $x$ is an eigenvector of the matrix $T$ associated with a Jordan block $J_p(\lambda)$, then its corresponding Jordan canonical basis $\gamma$, \emph{a cycle of generalized eigenvectors of $T$},  can be expressed as
$ \gamma = \left \{ v_p,v_{p-1},\ldots, v_1   \right \}$, where
 \begin{equation*}
 v_i = (T -\lambda I_{2k})^{p-i}(x)  \mbox{ for } i< p  \mbox{ and } v_p = x.
 \end{equation*}
 In this sense, we classify the possible eigenstructure types
 by using the notions defined in Lemma~\ref{Eigeven}
  as follows:\\
%
%
Case 1.
$c_{k,1}=0, Ce_1 = 0$. Then, $T \sim J_{k}(\lambda) \oplus J_{k}(\lambda)$ with two cycles of generalized eigenvectors taken as
\begin{eqnarray*}
\gamma_{1} & =&
 \left\{
\left[\begin{array}{c}  e_1 \\ 0 \end{array}\right],\ldots,
\left[\begin{array}{c} e_{k} \\ 0  \end{array}\right]
\right\},\\
\gamma_2 & =&
\left\{
\left[\begin{array}{c}  0 \\ e_1 \end{array}\right],
\left[\begin{array}{c} N_k Ce_2 \\ e_2  \end{array}\right],\ldots,
\left[\begin{array}{c} \sum\limits_{i=2}^k N_k^{k-i+1} Ce_i \\ e_k  \end{array}\right]
\right\}.
\end{eqnarray*}
%
%
Case 2. $c_{k,1}=0, Ce_1 \neq 0$. Then $T \sim J_{k}(\lambda) \oplus J_{k}(\lambda)$ with two cycles of generalized eigenvectors taken as
\begin{eqnarray*}
\gamma_1 & =&
\left\{
\left[\begin{array}{c}  e_1 \\ 0 \end{array}\right],\ldots,
\left[\begin{array}{c} e_{k} \\ 0  \end{array}\right]
\right\},\\
\gamma_2 & =& \left\{
\left[\begin{array}{c}  N_k Ce_1 \\ e_1 \end{array}\right],\ldots,
\left[\begin{array}{c} \sum\limits_{i=1}^k N_k^{k-i+1} Ce_i \\ e_k  \end{array}\right]
\right\}.
\end{eqnarray*}
%
%
Case 3. $c_{k,1}\neq 0$.   Then $T\sim J_{2k}(\lambda)$ with the cycle of generalized eigenvectors taken as
\begin{eqnarray*}
\gamma & =&
\left\{
\left[\begin{array}{c}  c_{k,1} e_1 \\ 0 \end{array}\right],\ldots,
\left[\begin{array}{c} c_{k,1} e_k \\ 0  \end{array}\right],\right.\\
&&\left.
\left[\begin{array}{c} N_k Ce_1 \\ e_1  \end{array}\right],
\ldots,
\left[\begin{array}{c}  \sum\limits_{i=1}^{k} N_k^{k-i+1} C e_i \\ e_k  \end{array}\right]
\right\}.
\end{eqnarray*}

Here, we have identified all possible eigenstructure types of the matrix $T$. This observation leads directly to the following result.

%
%
\begin{Theorem}\label{lemshiftAeven}
Let $\widehat{A}$ be the matrix defined by~\eqref{widhatAeven}.
Then,
the Jordan canonical form of the matrix $\widehat{A}$
is either $J_{k}(\lambda) \oplus J_{k}(\lambda)$ or  $J_{2k}(\lambda)$.
\end{Theorem}

We shall now give an example to demonstrate the result of Theorem~\ref{lem:shiftA} and Theorem~\ref{lemshiftAeven}.
\begin{example}\label{ex3}
Let $A=J_{4}(1) \oplus J_{2}(3) \in\mathbb{R}^{6\times 6}$ and
$U = \begin{bmatrix}
e_4 & e_3
\end{bmatrix}
$ and $V= \begin{bmatrix}
e_1 & e_2
\end{bmatrix}$ be the first half of the
generalized left and right eigenvectors of $A$ corresponding to the
eigenvalue $\lambda = 1$. We then want to
change the eigenvalue $\lambda=1$ to $\lambda=2$ and observe the subsequent Jordan canonical forms.

\begin{itemize}
\item[(a)] From~\eqref{eq:evenRU}, if we take the right and left inverses $R$ and $L$ to be
\[
R=V= \begin{bmatrix}
e_1 & e_2
\end{bmatrix},\,L=U= \begin{bmatrix}
e_4 & e_3
\end{bmatrix},
\]
then the shifted matrix $\widehat{A}$ satisfies
\[
   \widehat{A} :=
      A+(2-1)R_1R_2^*
      =J_{4}(2) \oplus J_{2}(3),
\]
where
$
R_1  =\begin{bmatrix}
V & L
\end{bmatrix}$ and
$R_2 =
\begin{bmatrix}
R & U
\end{bmatrix}.
$
\item[(b)]
From~\eqref{eq:evenRU}, if we take the right and left inverses $R$ and $L$ to be
\[
    R=\bb e_1 & e_2-e_3\eb,\,L=U=\bb e_4 & e_3\eb,
\]
then the shifted matrix $\widehat{A}$ satisfies
\[
     \widehat{A} :=  A+(2-1)R_1R_2^*=J_{2}(2) \oplus J_{2}(2)\oplus J_{2}(3),
\]
where
$
R_1  =\begin{bmatrix}
V & L
\end{bmatrix}$ and
$R_2 =
\begin{bmatrix}
R & U
\end{bmatrix}.
$
\end{itemize}
\end{example}

Note that Theorem~\ref{lemshiftAeven} also implies that  the geometric multiplicity of $\widehat{A}$, defined in~\eqref{widhatAeven},
 is at most two.
On the other hand,  an analogous approach can seemingly be used to derive a shift technique and  characterize the corresponding eigenstructure for the eigenvalue with odd algebraic multiplicities. Indeed, this analysis for the odd case is much more complicated due to the orthogonal property caused by the middle eigenvector as it can be seen in Theorem~\ref{thm:bio} and a different approach is needed for our discussion.

%
%

\section{Eigenvalue with Odd Algebraic Multiplicities}
Let us now consider the eigenvalue with odd algebraic multiplicities.
If $p = q$ and $p,q$ are odd integers, observe two full rank matrices
$U = \begin{bmatrix} u_1&\ldots & u_{p}\end{bmatrix}$
and $V = \begin{bmatrix} v_1&\ldots& v_{p}\end{bmatrix}$. It follows from~\eqref{eq:biothogonal} that the diagonal elements
of the lower triangular Hankel matrix $U^*V$ is constant. Thus, the inner product $u_{\frac{p+1}{2}}^* v_{\frac{p+1}{2}} \neq 0$. With this in mind, the
%
following result shows how to change an eigenvalue
with odd algebraic multiplicities.
\begin{Theorem}\label{lem:shiftAodd}
Let $A \in \mathbb{C}^{n\times n}$, and $\lambda_0$  be an eigenvalue of $A$ having algebraic multiplicity $2k+1$ and geometric multiplicity $1$, and let $\{u_i\}_{i=1}^{2k+1}$ and $\{v_i\}_{i=1}^{2k+1}$ be the left and right Jordan chains for $\lambda_0$.
Define two matrices
\begin{align}\label{oddUV}
U : =\begin{bmatrix}
u_1 & u_2 &\cdots & u_k
\end{bmatrix},\,
V : =\begin{bmatrix}
v_1 & v_2 &\cdots & v_k
\end{bmatrix}
\end{align}
and a vector $r = \dfrac{u_{k+1}}{ v_{k+1}^* u_{k+1}}$.
If matrices $R^*\in\mathbb{C}^{k\times n}$ and $L\in\mathbb{C}^{n\times k}$
are, respectively, the right and left inverses of $V$ and $U^*$ satisfying
\begin{align}\label{oddRL}
R^* V=U^* L=I_k,
\end{align}
then the shifted matrix
\begin{equation}\label{eq:widehatodd}
\widehat{A}  :=  A+(\lambda_1-\lambda_0)R_1R_2^*
\end{equation}
where
$
R_1 : =\begin{bmatrix}
V & v_{k+1} & L
\end{bmatrix}$ and
$R_2 :=\begin{bmatrix}
R & r & U
\end{bmatrix},
$ has the following properties:
\begin{itemize}
\item[(a)]
The eigenvalues of $\widehat{A}$ consist of those of $A$, except that the eigenvalue $\lambda_0$
of $A$ is  replaced by $\lambda_1$.
\item[(b)]
$\widehat{A}v_1=\lambda_1 v_1,\,\, u_1^* \widehat{A}=u_1^* \lambda_1, \,\,
\widehat{A}v_{i+1}=\lambda_1 v_{i+1}+v_i,\,\, u_{i+1}^* \widehat{A}=\lambda_1 u_{i+1}^*+u_i^*$,  for $i=1,\ldots,k-1$.That is to say,
\begin{align*}
\widehat{A} V &= VJ_{k}(\lambda_1)\\
U^* \widehat{A}&= J_{k}^\top(\lambda_1)U^*\\
 \end{align*}
\end{itemize}
\end{Theorem}
\begin{proof}
Since
\begin{eqnarray*}
R_2^* (A - \lambda  I)^{-1} R_1\!\!\!\! \!& = \!&\!\!\!\!
\left[\begin{array}{ccc} R^* (A - \lambda  I)^{-1} V  \!\!&\! R^*(A - \lambda  I_n)^{-1}v_{k+1} \!\!&\!
     R^* (A - \lambda  I)^{-1} L\\0 \!\!&\! \dfrac{1}{\lambda_0-\lambda} \!\!&\! r^* (A - \lambda  I)^{-1} L\\0 \!\!&\! 0 \!\!&\! U^* (A - \lambda  I)^{-1} L\end{array}\right],
\end{eqnarray*}
it follows from Theorem~\ref{thm:bio} and Theorem~\ref{lemA1} that the characteristic polynomial of  $\widehat{A}$ satisfies
\begin{eqnarray*}\label{charMhat1}
  \det(\widehat{A}-\lambda I_{n}) &=&
    \det(A- \lambda I_n)
    \det\left( I_n +
    (\lambda_1 - \lambda_0) R_1R_2^*
    (A - \lambda  I_n)^{-1}
    \right) \\
   & = &    \det(A- \lambda I_n)
    \det\left(
    I_{2k+1} +
    (\lambda_1 - \lambda_0) R_2^*
    (A - \lambda  I_n)^{-1}
R_1\right) \\
   & = &   \det(A- \lambda I_n) \left (
  \dfrac{\lambda_1 - \lambda}{\lambda_0 - \lambda }
   \right)^{2k+1},
 \end{eqnarray*}
if $\lambda\not\in\sigma(A)$. For the case $\lambda\in\sigma(A)$, a small perturbation of $\lambda$ is applied to make $A-\lambda I $ nonsingular. Thus, a similar proof like the one in Theorem~\ref{lem:shiftA} is followed. %
This proves (a).

We see by direct computation and Theorem~\ref{thm:bio} that
\begin{eqnarray*}
 \widehat{A}v_1 \!\!&\!\!=\!\!&\!\!
 Av_1+[(\lambda_1-\lambda_0)R_1R_2^* ] v_1
 = \lambda_0 v_1 + (\lambda_1 - \lambda_ 0) v_1
 = \lambda_1 v_1,\\
\widehat{A}v_{i+1} \!\!&\!\!=\!\!&\!\!\lambda_0 v_{i+1}+v_i+(\lambda_1-\lambda_0)R_1\begin{bmatrix}e_{i+1}\\0\end{bmatrix}=
\lambda_1 v_{i+1}+v_i, \quad  i = 1,\ldots, k-1.
\end{eqnarray*}
In an analogous way,  we can obtain $u_1^* \widehat{A}=u_1^* \lambda_1$ and $u_{i+1}^* \widehat{A}=\lambda_1 u_{i+1}^*+u_i^*$, for $i=1,\ldots,k-1$, and (b) follows. \hfill $\Box$
\end{proof}

Theorem~\ref{lem:shiftAodd} provides us a way to update the eigenvalue of a given Jordan block of odd size. Consequently, we are interested in analyzing possible eigenstructure types of the shifted matrix $\widehat{A}$. We start this discussion in an analogous way from Section~\ref{sec:even} by using the notions
defined in Theorem~\ref{lem:shiftAodd}. Assume for simplicity that
the given matrix $A$ is of size $(2k+1)\times (2k+1)$ and $P =
\bb v_1 & \ldots & v_{2k+1}\eb$ is a nonsingular matrix such that
\begin{align*}
P^{-1} {A} P=J_{2k+1}(\lambda_0).
\end{align*}
It follows that $V =P\begin{bmatrix}I_k \\ 0\end{bmatrix}$.
Set $U =P^{-*}\begin{bmatrix} 0\\ Q \end{bmatrix}$
for some lower triangular Hankel matrix $Q\in \mathbb{C}^{k\times k}$,
and the vector $v_{k+1}=P\hat{e}_{k+1}$, where $\hat{e}_{k+1}$
is a unit vector in $\mathbb{R}^{2k+1}$ with a 1 in position $k+1$ and $0$'s elsewhere.
Corresponding to formula~\eqref{oddRL} and the vector $r$ given in Theorem~\ref{lem:shiftAodd}, we define %
\begin{align*}
R =P^{-*}
\begin{bmatrix} I_k \\ S_1^*\end{bmatrix},\,
L =P\begin{bmatrix} S_2 \\ Q^{-*} \end{bmatrix},
\, r = P^{-*} \begin{bmatrix} 0 \\ w^*/\overline{w_1} \end{bmatrix},
\end{align*}
where $S_1^*$ and $S_2$ are arbitrary matrices in $\mathbb{C}^{(k+1)\times k}$, $w = [w_i]_{k\times 1}$ is a vector in $\mathbb{C}^{k+1}$, and $\overline{w_1}$ is the complex conjugate of $w_1$.
Then the shifted matrix~\eqref{eq:widehatodd} satisfies
\begin{align}\label{widhatAodd}
\widehat{A}&\! =  A+(\lambda_1-\lambda_0)(VR^*+v_{k+1}r^*+LU^*)\\
            &\! =  P\left[\!  J_{2k+1}(\lambda_0)+(\lambda_1-\lambda_0)\left(
            \begin{bmatrix}I_k\!&\!S_1\\ 0 \!&\!0 \end{bmatrix}
            + \hat{e}_{k+1}
            \begin{bmatrix} 0 \!&\! w/ {w_1}\end{bmatrix}
            +\begin{bmatrix}0  \!&\! S_2Q^{*}\\0  \!&\! I_k \end{bmatrix}\right)\!\right]P^{-1}\nonumber\\
          &\! =   P \begin{bmatrix}J_{k}(\lambda_1)\!&\!(\lambda_1-\lambda_0)\hat{s}_1 \!&\!(\lambda_1-\lambda_0)
            (S_1\begin{bmatrix}0 \\ I_k\end{bmatrix}
            +\begin{bmatrix} I_k \!&\! 0 \end{bmatrix} S_2Q^{*})
            \\
            0
            \!&\!\lambda_1\!&\!
            (\lambda_1-\lambda_0)(\hat{s}_2+
            \begin{bmatrix}
            w_2\!&\!\ldots\!&\!w_{k+1}
            \end{bmatrix})
              \!&\!\\
              0\!&\! 0\!&\!J_{k}(\lambda_1) \end{bmatrix}P^{-1},\nonumber
\end{align}
where $\hat{s}_1$ is the first column of $S_1$ and $\hat{s}_2$ is the last row of $S_2$.
This implies that after the shift approach the matrix $\widehat{A}$ is similar to an upper triangular matrix
 \begin{equation}\label{eq:similarodd}
S :=  \left[\begin{array}{ccc}J_k(\lambda)& a & C \\ 0 & \lambda & b^\top\\0&0& J_k(\lambda)\end{array}\right],
 \end{equation}
for some matrix $C\in\mathbb{C}^{k\times k}$ and vectors $a$, $b\in\mathbb{C}^{k}$.
This above similarity property allows us to discuss the eigeninformation of the shifted matrix $\widehat{A}$. That is, the eigeninformation of $\widehat{A}$ can be studied indirectly by considering the upper triangular matrix $T$ of~\eqref{eq:similarodd}

\begin{Lemma}
Let $C$ be a matrix in $\mathbb{C}^{k\times k}$ and let $a$, $b$ be two vectors in $\mathbb{C}^{k}$. If $S$ is given by
\begin{eqnarray*}
S &=& \left[\begin{array}{ccc}J_k(\lambda)& a & C \\ 0 & \lambda & b^\top\\0&0& J_k(\lambda)\end{array}\right],
\end{eqnarray*}
then $S$ has the eigenspace
\begin{align}\label{elambdaodd}
E_\lambda=
\left\{
\begin{array}{l}
\mbox{span}\left\{\bb e_1  \\0\\0\eb,\bb  N_k Ce_1 \\0\\e_1\eb\right\},
 \mbox{ if }
 b_1 = 0, a_k \neq  0,c_{k,1} =  0. \vspace{3mm}\\
\mbox{span}\left\{\bb e_1 \\0\\0\eb,\bb  N_k (\frac{-c_{k,1}}{a_k}a+Ce_1) \\ \frac{-c_{k,1}}{a_k}\\e_1\eb\right\},  \\
 \mbox{ if }
 b_1 = 0, a_k \neq 0,c_{k,1} \neq  0.  \vspace{3mm}\\
\mbox{span}\left\{\bb e_1  \\0\\0\eb,\bb N_k a \\1\\0\eb\right\}, \\
\mbox{ if } b_1 = 0,  a_k =0, c_{k,1} \neq 0
\mbox{ or } b_1 \neq 0,  a_k = 0.
 \vspace{3mm}\\
\mbox{span}\left\{\bb e_1 \\0\\0\eb,\bb  N_k(a+Ce_1 ) \\1\\e_1 \eb,\bb  N_k Ce_1  \\0\\e_1\eb \right\},
 \\
 \mbox{ if }
 b_1 = 0, a_k = 0,c_{k,1} =  0. \vspace{3mm}\\
\mbox{span}\left\{\bb e_1  \\0\\0\eb \right\}, \mbox{ if }
b_1 \neq 0,  a_k  \neq 0. \vspace{3mm}
\end{array}
\right.
\end{align}
 corresponding to the eigenvalue $\lambda$, where $N_k^\top= \lambda I_k - J_k(\lambda)$, $b_1 = e_1^\top b$,
 $a_k  = e_k^\top a$, and $C = [c_{i,j}]_{k\times k}$.
\end{Lemma}
\begin{proof}
Corresponding to the eigenvalue $\lambda$, let $v^\top = \bb x_1^\top & x_2^\top & x_3^\top\eb$ be an eigenvector of $S$, where
$x_1, x_3\in\mathbb{C}^{k}$ and $x_2\in \mathbb{C}$.
It is true that the eigenvector $v$ satisfies
\begin{align*}
S v = \lambda v,
\end{align*}
or, equivalently,
\begin{align*}
(J_k(\lambda)-\lambda I_k) x_1&=-x_2 a-  C x_3,\\
0 x_2         &=-  b^\top x_3 ,\\
(J_k(\lambda)-\lambda I_k) x_3&=0.
\end{align*}
We then have the all possible types of eigenspace by discussing
the cases given in~\eqref{elambdaodd} step by step.\hfill $\Box$
 \end{proof}

Similarly, we proceed to discuss all possible Jordan canonical forms of this shifted matrix~\eqref{eq:widehatodd} through the discussion of this special block upper triangular  matrix $S$ given in~\eqref{eq:similarodd}. One point should be made clear first. Due to the effects of vectors $a$ and $b$ in~$S$. It is too complicate to evaluate the generalized eigenvectors in an analogous way as we did in Section~\ref{sec:even}.  We, therefore, seek to determine the Jordan canonical forms by means of some kind of matrix transformation so that the structure and eigeninfomration of $S$ are simplified and unchanged respectively.

With this in mind, we start by choosing an invertible matrix
\begin{equation*}
Y=\bb I_k & y & W\\ 0 & 1&z^\top\\0 & 0 & I_k\eb
\in \mathbb{C}^{(2k+1)\times(2k+1)}.
\end{equation*}
It is easy to check that
\begin{equation*}
Y^{-1}=\bb I_k & -y & -W+y z^\top \\ 0 & 1& -z^\top\\0 & 0 & I_k\eb.
\end{equation*}
We then transform matrix $S$ in~\eqref{eq:similarodd} by using $Y$ and $Y^{-1}$ such that
\begin{align*}
Y S Y^{-1}=\bb J_k(\lambda) & a -(J_k(\lambda) - \lambda I)y & \widetilde{C}\\
0 & \lambda& b^\top+ z^\top(J_k(\lambda)-\lambda I)\\0 & 0 & J_k(\lambda)\eb,
\end{align*}
where
\begin{eqnarray*}
\widetilde{C} = [\tilde{c}_{i,j}]_{k\times k} =  W J_k(\lambda)-J_k(\lambda)W+C+yb^\top-az^\top+(J_k(\lambda)-\lambda I)y z^\top.
\end{eqnarray*}
Specify the vectors $a = [a_i]_{k\times 1} $ and $b = [b_i]_{k\times 1}$ by their components,
and  let the matrix $D =[d_{i,j}]_{k\times k} =C + yb^\top-az^\top+(J_k(\lambda)-\lambda I)y z^\top $.
%

Notice, first, that  if we choose
\begin{eqnarray*}
y^\top  =\bb y_1 &  a_1&  a_2& \ldots&  a_{k-1}\eb \mbox{ and }z^\top =\bb -b_2& -b_3& \ldots& -b_k& z_k\eb,
\end{eqnarray*}
 then for any $y_1$ and $z_k \in \mathbb{C}$ we have
\begin{equation}\label{eq:simple1}
\tilde{a} := a - (J_k(\lambda) - \lambda I) y =\bb   0 \\ a_k  \eb  \mbox{ and } \tilde{b}^\top := b^\top + z^\top \left( J_k(\lambda)-\lambda  I \right)=\bb b_1 & 0 \eb.
\end{equation}
Next, observe that
\begin{align*}
\left( W J_k(\lambda)-J_k(\lambda)W \right)e_i=
\left\{
\begin{array}{rl}
\left[\begin{array}{c}-w_{2,1} \\ \vdots \\ -w_{k,1} \\ 0\end{array}\right], &  i = 1, \smallskip\\
\left[\begin{array}{c}w_{1, i-1} - w_{2,i} \\ \vdots \\ w_{k-1,i-1} - w_{k,i} \\ w_{k,i-1}\end{array}\right], &
 2\leq i\leq k,
\end{array}
\right.
\end{align*}
where $ W = [w_{i,j}]_{k\times k}$.
We then want to eliminate the elements in $\widetilde{C}$ by
choosing
%
\begin{align*}
w_{i+1,1}&=d_{i,1},\quad 1\leq i \leq k-1,\\
w_{i+1,j}&=w_{i,j-1}+d_{i,j},\quad 1\leq i \leq k-1,\,2\leq j\leq k,
\end{align*}
so that
\begin{subequations}\label{eq:simple2}
\begin{eqnarray}
\tilde{c}_{i,j} &=&
 0,  \quad 1\leq i < k, \,1\leq j\leq k,\\
\tilde{c}_{k,j} & =&  \sum\limits_{\ell= 0}^{j-1} d_{k-\ell, j-\ell},\quad 1\leq j\leq k.
\end{eqnarray}
\end{subequations}
From formulae~\eqref{eq:simple1} and~\eqref{eq:simple2}, it follows that  by choosing appropriate vectors $y$, $z$ and a matrix $W$, we can simplify the matrix $S$
such that
\begin{align*}
S\sim \widetilde{S} := \left[\begin{array}{ccc}J_k(\lambda)& \tilde{a} & \widetilde{C} \\ 0 & \lambda & \tilde{b}^\top\\0&0& J_k(\lambda)\end{array}\right],
\end{align*}
where $\tilde{a}=a_k e_k$, $\tilde{b}=b_1 e_1$ are vectors in $\mathbb{C}^k$, and $\widetilde{C} = [\tilde{c}_{i,j}]_{k\times k}$ satisfies $\tilde{c}_{i,j} = 0$, for $1\leq i \leq k-1$ and $1\leq j\leq k$.
Note that $\widetilde{C}$ is a matrix with all entries equal to zero except the ones on the last row and is said to be in \emph{lower concentrated form}~\cite{HornJohnson91}.
This observation can be used to classify all possible eigenstructure types of $S$ though the discussion of $\widetilde{S}$.  We list all the
classifications as follows:\\
%
%
Case 1.
%
%
$b_1 \neq 0, a_k \neq  0$.
\begin{itemize}
\item[1a.]
If $\tilde{c}_{k,1}=\ldots=\tilde{c}_{k,i}=0$ and $\tilde{c}_{k,i+1}\neq 0$,
for $ 0 \leq i \leq k -1$,
then $S \sim J_{2k+1}(\lambda)$ with the cycle of generalized eigenvectors taken as
\begin{eqnarray*}
\gamma & =&
\left\{
\left[\begin{array}{c}b_1 a_k e_1 \\ 0 \\ 0\end{array}\right],\ldots,
\left[\begin{array}{c} b_1 a_k e_{k} \\ 0 \\ 0 \end{array}\right],
\left[\begin{array}{c}  0  \\ b_1  \\ 0
\end{array}\right],
\right.\\
&& \left.
\left[\begin{array}{c} 0 \\ 0 \\
e_1
\end{array}\right],\ldots,\left[\begin{array}{c} 0 \\ 0 \\
e_i
\end{array}\right],
\left[\begin{array}{c}
0 \\ \tau_1 \\
f_{1}
\end{array}
\right],\ldots,
\left[\begin{array}{c}
0 \\ \tau_{k-i} \\
f_{k-i}
\end{array}
\right]
\right\},\\
\end{eqnarray*}
where
\begin{eqnarray*}
f_{1} & = & e_{i+1} \in\mathbb{C}^k,\quad
f_{j}  =  e_{i+j} + \frac{1}{b_1} \sum\limits_{s= 1}^{j-1}
\tau_s e_{j-s} \in\mathbb{C}^k, \quad  2\leq j\leq k-i,\\
\tau_j & = & \dfrac{- e_k^\top \widetilde{C} f_j
}{a_k}\in\mathbb{C}, \quad 1 \leq j \leq k-i.
\end{eqnarray*}
\item
[1b.]
If  $\tilde{c}_{k,1}=\ldots=\tilde{c}_{k,k}=0$, then  $S\sim J_{2k+1}(\lambda)$ with the cycle of generalized eigenvectors taken as

\begin{eqnarray*}
\gamma & =&
\left\{
\left[\begin{array}{c}b_1 a_k e_1 \\ 0 \\ 0\end{array}\right],\ldots,
\left[\begin{array}{c} b_1 a_k e_{k} \\ 0 \\ 0 \end{array}\right],
\left[\begin{array}{c}  0  \\ b_1  \\ 0
\end{array}\right],
\left[\begin{array}{c} 0 \\ 0 \\
e_1
\end{array}\right],\ldots,\left[\begin{array}{c} 0 \\ 0 \\
e_k
\end{array}\right] \right\}.
\end{eqnarray*}

\end{itemize}
%
%
Case 2.
$b_1 = 0, a_k \neq  0$.
\begin{itemize}
\item[2a.]
If $\tilde{c}_{k,1}=\ldots=\tilde{c}_{k,k-1}=0$ and $\tilde{c}_{k,k}\neq 0$,
then $S\sim J_{k+1}(\lambda)\oplus J_{k}(\lambda)$ with the cycles of generalized eigenvectors taken as
\begin{eqnarray*}
\gamma_{1} \!&\! =\!&\!
\left\{
\left[\begin{array}{c} a_{k}e_1 \\ 0 \\ 0\end{array}\right],\ldots,\left[\begin{array}{c} a_{k}e_{k} \\ 0 \\ 0 \end{array}\right],
\left[\begin{array}{c} 0 \\ 1 \\ 0 \end{array}\right]
\right\},\\
\gamma_{2}       \!&\!=\!&\!
\left\{
\left[\begin{array}{c}0 \\ 0 \\e_1\end{array}\right],\ldots,
\left[\begin{array}{c}0 \\ 0 \\e_{k-1}\end{array}\right],
\left[\begin{array}{c}0 \\ \frac{-\tilde{c}_{k,k}}{{a}_k} \\e_{k}\end{array}\right]
\right\}.\\
\end{eqnarray*}

\item[2b.]
If $\tilde{c}_{k,1}=\ldots=\tilde{c}_{k,i}=0$ and $\tilde{c}_{k,i+1}\neq 0$,
for $ 0 \leq i  < k -1$,
then $S\sim J_{2k-i}(\lambda)\oplus J_{i+1}(\lambda)$ with the cycles of generalized eigenvectors taken as
\begin{eqnarray*}
\gamma_{1} \!&\! =\!&\!
\left\{
\left[\begin{array}{c}m_1 \\n_1\end{array}\right],\ldots,
\left[\begin{array}{c}m_{2k-i} \\n_{2k-i} \end{array}\right]
\right\},\\
\gamma_{2}       \!&\!=\!&\!
\left\{
\left[\begin{array}{c}0 \\ 0 \\e_1\end{array}\right],\ldots,
\left[\begin{array}{c}0 \\ 0 \\e_i\end{array}\right],
\left[\begin{array}{c}0 \\ \frac{-\tilde{c}_{k,i+1}}{{a}_k} \\e_{i+1}\end{array}\right]
\right\}.\\
\end{eqnarray*}
where
\begin{eqnarray*}
m_j &=& \left\{
\begin{array}{rl}
\tilde{c}_{k,i+1}e_j, & 1\leq j \leq k,
\smallskip \\
0_{k\times 1}, & k+1 \leq j \leq 2k -i .
\end{array}
\right.
\\
n_j &=& \left\{
\begin{array}{ll}
  0_{(k+1)\times 1}, &  1\leq j < k - i + 1,
  \smallskip \\
  \bb 0 &   \sum\limits_{s= 0}^{j-k+i-1}
\alpha_s e_{j-k+i-s}^\top \eb^\top, &  k-i+1\leq j < 2k -2 i- 1,
\smallskip \\
\bb 0 &   \sum\limits_{s= 0}^{k-i-2}
\alpha_s e_{j-k+i-s}^\top \eb^\top, &  2k -2 i- 1\leq j \leq 2k - i- 1,
\smallskip \\
\bb \dfrac{-1}{a_k} \sum\limits_{s = {0}}^{k-i-2}
 \alpha_{s} \tilde{c}_{k, k-s} &   \sum\limits_{s= 0}^{k-i-2}
\alpha_s e_{k-s}^\top \eb^\top, &  j = 2k - i,
\end{array}
\right.
\\
\alpha_0 & = & 1,\quad
\alpha_{j}  = \dfrac{-1}{\tilde{c}_{k,i+1}} \sum\limits_{s = {0}}^{j-1}
 \alpha_{s} \tilde{c}_{k, i+2+s}\in\mathbb{C}, \quad 1\leq j \leq k-i-1.
\end{eqnarray*}
Here, if $k = i+2$, we should ignore the second case define in $n_j$
by using the third one directly.

\item
[2c.] If $\tilde{c}_{k,1}=\ldots=\tilde{c}_{k,k}=0$, then  $S\sim J_{k+1}(\lambda)\oplus J_{k}(\lambda)$
with the cycles of generalized eigenvectors taken as
\begin{eqnarray*}
\gamma_1       &=&
\left\{
\left[\begin{array}{c} a_k e_1 \\ 0 \\ 0 \end{array}\right],\ldots,
\left[\begin{array}{c} a_k e_k \\  0 \\0\end{array}\right],
\left[\begin{array}{c} 0  \\  1 \\0\end{array}\right]
\right
\},\\
\gamma_2       &=&
\left\{
\left[\begin{array}{c}0 \\ 0 \\e_1\end{array}\right],\ldots,
\left[\begin{array}{c}0 \\  0 \\e_{k}\end{array}\right]
\right\}.
\end{eqnarray*}
\end{itemize}
 %
 %
Case 3.$ b_1 \neq 0, a_k =  0$.
\begin{itemize}
\item[3a.]
If $\tilde{c}_{k,1}=\ldots=\tilde{c}_{k,i}=0$ and $\tilde{c}_{k,i+1}\neq 0$,
for $ 0 \leq i \leq k -1$,
then $S\sim J_{2k-i}(\lambda)\oplus J_{i+1}(\lambda)$
with the cycles of generalized eigenvectors taken as
\begin{eqnarray*}
\gamma_1 & =&
\left\{
\left[\begin{array}{c}m_1 \\n_1\end{array}\right],\ldots,
\left[\begin{array}{c}m_{2k-i} \\n_{2k-i} \end{array}\right]
\right\},\\
\gamma_2       &=&
\left\{
\left[\begin{array}{c}0 \\{b}_1 \\  0 \end{array}\right],
\left[\begin{array}{c}0 \\ 0 \\e_1\end{array}\right],\ldots,
\left[\begin{array}{c}0 \\ 0 \\e_i\end{array}\right]
\right\}.\\
\end{eqnarray*}
where
\begin{eqnarray*}
m_j &=& \left\{
\begin{array}{rl}
\tilde{c}_{k,i+1}e_j, & 1\leq j \leq k,
\smallskip \\
0_{k\times 1}, & k+1 \leq j \leq 2k -i .
\end{array}
\right.
\\
n_j &=& \left\{
\begin{array}{ll}
  0_{(k+1)\times 1}, &  1\leq j < k - i ,
  \smallskip \\
   \left[\begin{array}{cc} b_1 &   0\end{array}\right]^\top  , &   j =  k - i  ,
  \smallskip \\
  \left[\begin{array}{cc}b_1\alpha_{j-k+i} &   \sum\limits_{s= 0}^{j-k+i-1}
\alpha_s e_{j-k+i-s}^\top\end{array}\right]^\top , &  k-i+1\leq j <  2k -2 i,
\smallskip \\
  \left[\begin{array}{cc} 0  &    \sum\limits_{s= 0}^{k-i-1}
\alpha_s e_{j-k+i-s}\end{array}\right]^\top, &  2k -2 i\leq j \leq 2k - i,
\end{array}
\right.
\\
\alpha_0 & = & 1,\quad
\alpha_{j}  = \dfrac{-1}{\tilde{c}_{k,i+1}} \sum\limits_{s = {0}}^{j-1}
 \alpha_{s} \tilde{c}_{k,i+2+s}\in\mathbb{C}, \quad 1 \leq j\leq k-i-1.
\end{eqnarray*}
Here, if $k = i+1$, we should ignore
the first case and replace
the third case define in $n_j$
by using the forth one directly.
\item
[3b.] If $\tilde{c}_{k,1}=\ldots=\tilde{c}_{k,k}=0$, then  $S\sim J_{k+1}(\lambda)\oplus J_{k}(\lambda)$
with the cycles of generalized eigenvectors taken as
\begin{eqnarray*}
\gamma_1      &=&
\left\{
\left[\begin{array}{c}0 \\ {b}_1 \\0\end{array}\right]
\left[\begin{array}{c}0 \\ 0 \\e_1\end{array}\right],\ldots,
\left[\begin{array}{c}0 \\  0 \\e_{k}\end{array}\right]
\right
\},\\
\gamma_2      &=&
\left\{
\left[\begin{array}{c} e_1 \\ 0 \\ 0 \end{array}\right],\ldots,
\left[\begin{array}{c} e_k \\  0 \\0\end{array}\right]
\right\}.
\end{eqnarray*}
\end{itemize}
%
%
Case 4. $\tilde{b}_1 = 0, \tilde{a}_k =  0$.
%
%
%
%
\begin{itemize}
\item[4a.] If $\tilde{c}_{k,1}\neq 0$, then  $S\sim J_{2k}(\lambda)\oplus J_{1}(\lambda)$
with the cycles of generalized eigenvectors taken as
\begin{eqnarray*}
\gamma_1       &=&
\left\{
\left[\begin{array}{c} \tilde{c}_{k,1} e_1 \\ 0 \\0\end{array}\right],\ldots,
\left[\begin{array}{c} \tilde{c}_{k,1} e_k \\ 0 \\ 0\end{array}\right],
\right.\\
&&
\left.
 \left[\begin{array}{c}0 \\  0 \\\psi_1e_{1}\end{array}\right],\ldots,
\left[\begin{array}{c}0 \\  0 \\ \sum\limits_{s=1}^k \psi_s e_{k-s+1}\end{array}\right]
\right
\},\\
\gamma_2      &=&
\left\{
\left[\begin{array}{c} 0 \\ 1 \\ 0 \end{array}\right]
\right\}.
\end{eqnarray*}
where
\begin{eqnarray*}
\psi_1 & = & 1,\quad
\psi_{j} = \dfrac{-1}{\tilde{c}_{k,1}} \sum\limits_{s = {1}}^{j-1}
 \psi_{s} \tilde{c}_{k,j-s+1}\in\mathbb{C}, \mbox{ for } j = 2,\ldots, k.
\end{eqnarray*}
\item[4b.]
If  $\tilde{c}_{k,1}=\ldots={\tilde{c}}_{k,i}=0$ and $\tilde{c}_{k,i+1}\neq 0$,
for $ 1 \leq i \leq k -1$,
then $S\sim J_{2k-i}(\lambda) \oplus J_{i}(\lambda) \oplus J_1(\lambda)$,
with the cycles of generalized eigenvectors taken as
\begin{eqnarray*}
\gamma_1 & =&
\left\{
\left[\begin{array}{c} m_1 \\ 0 \\ n_1\end{array}\right],\ldots,
\left[\begin{array}{c} m_{2k-i} \\ 0 \\ n_{2k-i}\end{array}\right]
%
\right\},\\
\gamma_2    &=&
\left\{
\left[\begin{array}{c}0 \\ 0 \\e_1\end{array}\right],\ldots,
\left[\begin{array}{c}0 \\ 0 \\e_i\end{array}\right]
\right\},\\
\gamma_3      &=&
\left\{
\left[\begin{array}{c}0 \\ 1 \\  0 \end{array}\right]
\right\},
\end{eqnarray*}
where
\begin{eqnarray*}
m_j &=& \left\{
\begin{array}{rl}
\tilde{c}_{k,i+1}e_j, & 1\leq j \leq k,
\smallskip \\
0_{k\times 1}, & k+1 \leq j \leq 2k -i .
\end{array}
\right.
\\
n_j &=& \left\{
\begin{array}{ll}
  0_{(k+1)\times 1}, &  1\leq j < k - i + 1,
  \smallskip \\
   \sum\limits_{s= 0}^{j-k+i-1}
\alpha_s e_{j-k+i-s}, &  k-i+1\leq j < 2k -2 i,
\smallskip \\
   \sum\limits_{s= 0}^{k-i-1}
\alpha_s e_{j-k+i-s} , &  2k -2 i\leq j \leq 2k - i,
\end{array}
\right.
\\
\alpha_0 & = & 1,\quad
\alpha_{j} = \dfrac{-1}{\tilde{c}_{k,i+1}} \sum\limits_{s = {0}}^{j-1}
 \alpha_{s} \tilde{c}_{k,i+2+s}\in\mathbb{C}, \mbox{ for } j = 1,\ldots, k- i -1.
\end{eqnarray*}
Here, if $k = i+1$, we should ignore the second case define in $n_j$
by using the third one directly.

\item
[4c.] If $\tilde{c}_{k,1}=\ldots=\tilde{c}_{k,k}=0$, then  $A\sim J_{k}(\lambda)\oplus J_{k}(\lambda)\oplus J_{1}(\lambda)$,
with the cycles of generalized eigenvectors taken as
\begin{eqnarray*}
\gamma_1      &=&
\left\{
\left[\begin{array}{c}0 \\ 0 \\e_1\end{array}\right],\ldots,
\left[\begin{array}{c}0 \\  0 \\e_{k}\end{array}\right]
\right
\},\\
\gamma_2         &=&
\left\{
\left[\begin{array}{c} e_1 \\ 0 \\ 0 \end{array}\right],\ldots,
\left[\begin{array}{c} e_k \\  0 \\0\end{array}\right]
\right\},\\
\gamma_3        &=&
\left\{
\left[\begin{array}{c}0 \\ 1 \\  0 \end{array}\right]
\right\}.
\end{eqnarray*}
\end{itemize}

%
%
\begin{Theorem}\label{lemshiftAodd}
Let $\widehat{A}$ be the matrix defined by~\eqref{widhatAodd}.
Then,
the Jordan canonical form of the matrix $\widehat{A}$
is one of the types:
$J_{2k+1}(\lambda)$,
$J_{2k-i+1}(\lambda) \oplus J_{i}(\lambda)$, and
$J_{2k-i}(\lambda) \oplus J_{i}(\lambda)\oplus J_{1}(\lambda)$,
for $1\leq i\leq k$.
\end{Theorem}

Now, we have all  possible eigenstructure types of the matrix $S$.
In view of such classifications above, it turns out that the  geometric multiplicity of $\widehat{A}$, defined in~\eqref{eq:widehatodd},  is at most three.

\section{Conclusions and Open Problems}
The methods developed in this paper are used to shift an eigenvalue with algebraic multiplicities greater than $1$, in the sense that the first half of corresponding generalized eigenvectors are kept unchanged. From the point of view of applications, the approach has the advantage that one can apply this shift technique to speed up or stabilizing a given numerical algorithm.

It is true that there are many different kinds of eigenvalue shift problems for a wide range of applications in science and engineering.
For example, in our recent work~\cite{Lin20135083},
we apply the shift approach to remove two zero eigenvalues
embedded in a nonsymmetric algebraic Riccati equation. After the shift, the speed of convergence of the simple iteration algorithm for finding the minimal nonnegative solution is significantly improved.  Indeed, this application is a special case corresponding to Theorem~\ref{lem:shiftA} with $k=1$.

That is, in this work, we propose a more general way to shift eigenvalues of a given matrix with multiple algebraic multiplicities. But, it is worthy of note that the computation of the Jordan form is extremely sensitive to perturbations. It appears to be an interesting open and challenge problem to propose a reliably numerical method to compute the Jordan form of a given matrix with floating-point arithmetic.
On the other hand, in most cases we are required to
shift partial (not all) eigenvalues of a given Jordan block,
change multiple eigenvalues simultaneously, or  replace complex conjugate eigenvalues of a real matrix.
All these questions are under investigation and will be reported elsewhere.

%


\section*{Acknowledgement}
%
 This research work is partially supported by the National Science Council and
the National Center for Theoretical Sciences in Taiwan.

%

\end{document}